\batchmode
\makeatletter
\def\input@path{{\string"/Users/russw/Documents/Research/mypapers/Matchings, coverings, and Castelnuovo-Mumford regularity/\string"/}}
\makeatother
\documentclass[12pt,oneside,english]{amsart}
\usepackage[T1]{fontenc}
\usepackage[latin9]{inputenc}
\usepackage{geometry}
\usepackage{babel}
\usepackage{verbatim}
\usepackage{amsthm}
\usepackage{amssymb}
\usepackage[unicode=true,pdfusetitle,
 bookmarks=true,bookmarksnumbered=false,bookmarksopen=false,
 breaklinks=false,pdfborder={0 0 1},backref=false,colorlinks=false]
 {hyperref}

\makeatletter
\theoremstyle{plain}
\newtheorem{thm}{\protect\theoremname}
  \theoremstyle{remark}
  \newtheorem{rem}[thm]{\protect\remarkname}
  \theoremstyle{plain}
  \newtheorem{lem}[thm]{\protect\lemmaname}
  \theoremstyle{plain}
  \newtheorem{prop}[thm]{\protect\propositionname}
  \theoremstyle{plain}
  \newtheorem{cor}[thm]{\protect\corollaryname}
  \theoremstyle{definition}
  \newtheorem{example}[thm]{\protect\examplename}
  \theoremstyle{plain}
  \newtheorem{question}[thm]{\protect\questionname}

\usepackage{ae}
\usepackage{aecompl}

\makeatother

  \providecommand{\corollaryname}{Corollary}
  \providecommand{\examplename}{Example}
  \providecommand{\lemmaname}{Lemma}
  \providecommand{\propositionname}{Proposition}
  \providecommand{\questionname}{Question}
  \providecommand{\remarkname}{Remark}
\providecommand{\theoremname}{Theorem}

\begin{document}

\global\long\def\normalin{\mathrel{\lhd}}

\global\long\def\innormal{\mathrel{\rhd}}

\global\long\def\semidirect{\mathbin{\rtimes}}

\global\long\def\Stab{\operatorname{Stab}}

\global\long\def\bdry{\partial}

\global\long\def\susp{\operatorname{susp}}

\global\long\def\lrprod{\mathop{\check{\prod}}}

\global\long\def\lrtimes{\mathbin{\check{\times}}}

\global\long\def\urtimes{\mathbin{\hat{\times}}}

\global\long\def\urprod{\mathop{\hat{\prod}}}

\global\long\def\subsetdot{\mathrel{\subset\!\!\!\!{\cdot}\,}}

\global\long\def\dotsupset{\mathrel{\supset\!\!\!\!\!\cdot\,\,}}

\global\long\def\precdot{\mathrel{\prec\!\!\!\cdot\,}}

\global\long\def\dotsucc{\mathrel{\cdot\!\!\!\succ}}

\global\long\def\des{\operatorname{des}}

\global\long\def\rank{\operatorname{rank}}

\global\long\def\height{\operatorname{height}}

\global\long\def\modreln{\mathrel{M}}

\global\long\def\link{\operatorname{link}}

\global\long\def\freejoin{\mathbin{\circledast}}

\global\long\def\stellarsd{\operatorname{stellar}}

\global\long\def\conv{\operatorname{conv}}

\global\long\def\disjointunion{\mathbin{\dot{\cup}}}

\global\long\def\skel{\operatorname{skel}}

\global\long\def\depth{\operatorname{depth}}

\global\long\def\st{\operatorname{star}}

\global\long\def\alexdual#1{#1^{\vee}}

\global\long\def\reg{\operatorname{reg}}

\global\long\def\shift{\operatorname{Shift}}

\global\long\def\Dom{\operatorname{Dom}}

\global\long\def\cosetposet{\overline{\mathfrak{C}}}

\global\long\def\cosetlat{\mathfrak{C}}

\global\long\def\im{\operatorname{indmatch}}

\global\long\def\ind{\operatorname{Ind}}

\global\long\def\cochordcov{\operatorname{cochord}}

\global\long\def\boxicity{\operatorname{box}}

\global\long\def\edgeideal{\mathtt{I}}

\title[Matchings, coverings, and Castelnuovo-Mumford regularity]{Matchings, coverings,\linebreak{}
 and Castelnuovo-Mumford regularity }

\author{Russ Woodroofe}

\address{Department of Mathematics \& Statistics, Mississippi State University,
MS 39762}

\email{rwoodroofe@math.msstate.edu}

\begin{abstract}
We show that the co-chordal cover number of a graph $G$ gives an
upper bound for the Castelnuovo-Mumford regularity of the associated
edge ideal. Several known combinatorial upper bounds of regularity
for edge ideals are then easy consequences of covering results from
graph theory, and we derive new upper bounds by looking at additional
covering results.\vspace{-0.15in}

\end{abstract}
\maketitle

\section{\label{sec:Introduction}Introduction and background}

Let $G$ be a graph with vertex set $\{x_{1},\dots,x_{n}\}$, and
let $R=k[x_{1},\dots,x_{n}]$ be the polynomial ring over a field
$k$ obtained by associating a variable with each vertex of $G$.
We consider the \emph{edge ideal of $G$ in $R$}, defined as $\edgeideal(G)=\left(x_{i}x_{j}\,:\,\{x_{i},x_{j}\}\mbox{ an edge of }G\right)$. 

The Castelnuovo-Mumford regularity of an ideal $I$, denoted by $\reg I$,
is one of the main measures of the complexity of $I$. Several recent
papers \cite{Francisco/Ha/VanTuyl:2009,Ha/VanTuyl:2008,Moradi/Kiani:2010,Nevo:2011,Terai:2001,VanTuyl:2009,Whieldon:2010UNP}
have related the Castelnuovo-Mumford regularity of the edge ideal
$\edgeideal(G)$ with various invariants of the graph $G$. 

The purpose of this paper is to give a new upper bound on $\reg\left(R/\edgeideal(G)\right)$,
and to show that this new upper bound generalizes several other recently
discovered upper bounds. 

A graph $G$ is \emph{chordal} if every induced cycle in $G$ has
length 3, and is \emph{co-chordal} if the complement graph $\overline{G}$
is chordal. It follows from Fröberg's classification of edge ideals
with linear resolutions \cite{Froberg:1990} that $\reg\left(R/\edgeideal(G)\right)\leq1$
if and only if $G$ is co-chordal. (A direct proof using the techniques
in Section \ref{sec:Lower-bounds} is also straightforward). The \emph{co-chordal
cover number}, denoted\emph{ }$\cochordcov G$, is the minimum number
of co-chordal subgraphs required to cover the edges of $G$. 

Our main result is as follows:
\begin{thm}
\label{thm:MainLemma} For any graph $G$ and over any field $k$,
we have $\reg\left(R/\edgeideal(G)\right)\leq\cochordcov G$. 
\end{thm}
We will see the proof to follow almost immediately from a result of
Kalai and Meshulam \cite{Kalai/Meshulam:2006}. Nevertheless, Theorem
\ref{thm:MainLemma} provides a fundamental connection between combinatorics
and commutative algebra, and it will help us give simple and unified
proofs of both known and new upper bounds for the regularity of $R/\edgeideal(G)$. 

A particularly simple condition yielding a co-chordal cover (hence
a bound on regularity) is as follows:
\begin{thm}
\label{thm:SplitCovering}If $G$ is a graph such that $V(G)$ can
be partitioned into an (induced) independent set $J_{0}$ together
with $s$ cliques $J_{1},\dots,J_{s}$, then $\reg\left(R/\edgeideal(G)\right)\leq s$.
\end{thm}
The following is a recursive version of Theorem \ref{thm:SplitCovering}: 
\begin{thm}
\label{thm:SplitRecursiveFormulation}If $G$ is a graph such that
$J\subseteq V(G)$ induces a clique, then 
\[
\reg\left(R/\edgeideal(G)\right)\leq\reg\left(R/\edgeideal(G\setminus J)\right)+1,
\]
where $G\setminus J$ denotes the induced subgraph on $V(G)\setminus J$.
\end{thm}
\noindent In plain language, Theorem \ref{thm:SplitRecursiveFormulation}
says that deleting a clique lowers regularity by at most 1. The author
hopes that Theorems \ref{thm:SplitCovering} and \ref{thm:SplitRecursiveFormulation}
may be helpful to practitioners in the field for quickly finding rough
upper estimates of regularity of edge ideals.

\medskip{}

The remainder of this paper is organized as follows. In the remainder
of this section we review terminology from graph theory. In Section~\ref{sec:ProofMainLemma},
we prove Theorem \ref{thm:MainLemma}. In Section~\ref{sec:Lower-bounds},
we introduce the equivalent notion of regularity of a simplicial complex.
We then use topological techniques to calculate regularity of several
examples, and more generally to obtain lower bounds. In particularly
we give a geometric proof of the well-known fact (Lemma \ref{lem:InducedMatchingLB})
that $\reg\left(R/\edgeideal(G)\right)$ is at least the induced matching
number of $G$. In Section \ref{sec:Applications}, we combine Theorem
\ref{thm:MainLemma} with results from the graph theory literature
to prove Theorems \ref{thm:SplitCovering} and \ref{thm:SplitRecursiveFormulation}.
We recover and extend results of \cite{Ha/VanTuyl:2008} and \cite{Kummini:2009},
but show that results of \cite{Mahmoudi/Mousivand/Crupi/Rinaldo/Terai/Yassemi:2011}
and \cite{VanTuyl:2009} cannot be proved using this technique.

\subsection{\label{sub:Notation}Terminology and notation from graph theory}

All graphs discussed in this paper are simple, with no loops or multiedges.
We assume basic familiarity with standard graph theory definitions
as in e.g. \cite{Diestel:2005} or \cite{Mahadev/Peled:1995}, but
review some particular terms we will use:

If $\mathcal{F}$ is a family of graphs, then an \emph{$\mathcal{F}$
covering} of a graph $G$ is a collection $H_{1},\dots,H_{s}$ of
subgraphs of $G$ such that every $H_{i}$ is in $\mathcal{F}$, and
such that $\bigcup E(H_{i})=E(G)$. Elsewhere in the literature this
notion is sometimes referred to as an \emph{$\mathcal{F}$ edge covering},
to contrast with covers of the vertices. The \emph{$\mathcal{F}$
cover number} is the smallest size of an $\mathcal{F}$ cover. We
will mostly be interested in the case where $\mathcal{F}$ is some
subfamily of co-chordal graphs.

An \emph{independent set }in a graph $G$ is a subset of pairwise
non-adjacent vertices. Similarly a \emph{clique} is a subset of pairwise
adjacent vertices. We do not require cliques to be maximal. 

A \emph{matching} in a graph $G$ is a subgraph consisting of pairwise
disjoint edges. If the subgraph is an induced subgraph, the matching
is an \emph{induced matching}. The graph consisting of a matching
with $m$ edges we denote as $mK_{2}$. 

The \emph{independence number} $\alpha(G)$, \emph{clique number}
$\omega(G)$, and \emph{induced matching number $\im G$} are respectively
the maximum size of an independent set, clique, or induced matching. 

A \emph{coloring} of $G$ is a partition of the vertices into (induced)
independent sets (\emph{colors}), and the \emph{chromatic number}
$\chi(G)$ is the smallest number of colors possible in a coloring
of $G$. A graph $G$ is \emph{perfect} if $\alpha(H)=\chi(\overline{H})$
for every induced subgraph $H$ of $G$. It is well-known that the
complement of a perfect graph is also perfect.

We denote by $P_{n}$ the path on $n$ vertices (having edges $\{x_{1}x_{2},x_{2}x_{3},\dots,x_{n-1},x_{n}\}$),
and by $C_{n}$ the cycle on $n$ vertices (having the edges of $P_{n}$
together with $x_{1}x_{n}$).

\section{\label{sec:ProofMainLemma}Proof of Theorem \ref{thm:MainLemma}}

As previously mentioned, Theorem \ref{thm:MainLemma} is an easy consequence
of the following deep result of Kalai and Meshulam \cite{Kalai/Meshulam:2006}.
\begin{thm}
\emph{\label{thm:KalaiMeshulam}(Kalai and Meshulam \cite[Theorem 1.2]{Kalai/Meshulam:2006})
If $I_{1},\dots,I_{s}$ are square-free monomial ideals of a polynomial
ring $R=k[x_{1},\dots,x_{n}]$ (for some field $k$), then 
\[
\reg\left(R\,\big/\left(I_{1}+\dots+I_{s}\right)\right)\leq\sum_{j=1}^{s}\reg\left(R/I_{j}.\right)
\]
}\end{thm}
\begin{rem}
Theorem \ref{thm:KalaiMeshulam} was conjectured by Terai \cite{Terai:2001}.
Herzog \cite{Herzog:2007} later generalized the result to monomial
ideals that are not square-free.
\end{rem}

\begin{rem}
Kalai and Meshulam stated \emph{\cite[Theorem 1.2]{Kalai/Meshulam:2006}}
in terms of $\reg(I_{j})$'s, rather than $\reg(R/I_{j})$'s. Theorem
\ref{thm:KalaiMeshulam} is equivalent, since by e.g. \cite[Theorem 1.34]{Miller/Sturmfels:2005},
we have $\reg I=\reg(R/I)+1$.
\end{rem}
In the context of edge ideals, Theorem \ref{thm:KalaiMeshulam} says
that if $G_{1},\dots,G_{s}$ are graphs on the same vertex set $\{x_{1},\dots,x_{n}\}$,
then 
\begin{equation}
\reg\left(R/\edgeideal(\bigcup_{j=1}^{s}G_{j})\right)\leq\sum_{j=1}^{s}\reg\left(R/\edgeideal(G_{j})\right).\label{eq:Kalai-Meshulam_EdgeIdeal}
\end{equation}

\begin{proof}[Proof of Theorem \ref{thm:MainLemma}.]
 Recall from above that $\reg\left(R/\edgeideal(H)\right)=1$ if
and only if $H$ is co-chordal with at least one edge. The result
then follows immediately from (\ref{eq:Kalai-Meshulam_EdgeIdeal})
by considering the case where each $R/\edgeideal(G_{j})$ has regularity
1.
\end{proof}
We comment that (\ref{eq:Kalai-Meshulam_EdgeIdeal}) can more generally
be applied to edge ideals of clutters (i.e., to square-free monomial
ideals with degree $>2$), but that in this case the set of ideals
with linear resolution (that is, smallest possible regularity) is
not classified, giving more fragmented results. In this paper we henceforth
restrict ourselves to the case of graphs.

\section{\label{sec:Lower-bounds}Lower bounds and simple examples}

Before discussing applications, it will be convenient to have lower
bounds to compare with the upper bound of Theorem \ref{thm:MainLemma}.
As we will shortly see that $\reg\left(R/\edgeideal(H)\right)\leq\reg\left(R/\edgeideal(G)\right)$
for every induced subgraph $H$ of $G$, lower bounds usually come
from examples. 

We will compute regularity through Hochster's Formula (see e.g. \cite{Miller/Sturmfels:2005}),
which relates local cohomology of the quotient $R/I$ of a square-free
monomial ideal with the simplicial cohomology of the simplicial complex
of non-zero square-free monomials in $R/I$. We refer to \cite{Hatcher:2002}
for basic background on simplicial cohomology, or to \cite{Bjorner:1995}\textbf{
}for a concise reference aimed at combinatorics. 

The \emph{Castelnuovo-Mumford regularity }of a simplicial complex
$\Delta$\emph{ }over a field $k$, denoted $\reg_{k}\Delta$, is
defined to be the maximum $i$ such that the reduced homology $\tilde{H}_{i-1}(\Gamma;k)\neq0$
for some induced subcomplex $\Gamma$ of $\Delta$. It is well-known
to follow from Hochster's Formula (together with the Betti number
characterization of regularity) that $\reg_{k}\Delta$ is equal to
the Castelnuovo-Mumford regularity of the Stanley-Reisner ring of
$\Delta$ over $k$. We remark that complexes with regularity at most
$d$ have been referred to as \emph{$d$-Leray}, and have been studied
in the context of proving certain Helly-type theorems \cite{Kalai/Meshulam:2006}. 

In the case of the edge ideal of a graph $G$, let $\ind G$ denote
the \emph{independence complex of $G$}, consisting of all independent
sets of $G$. In this case our above discussion specializes to the
relation: 
\begin{equation}
\reg\left(k[x_{1},\dots,x_{n}]\big/\edgeideal(G)\right)=\reg_{k}\left(\ind G\right).\label{eq:TopReg}
\end{equation}
(Note that we write $k[x_{1},\dots,x_{n}]$ rather than $R$ to emphasize
the field over which we are working.) 

In particular, it follows immediately from definition of $\reg_{k}\Delta$
that $\reg_{k}\left(\ind H\right)\leq\reg_{k}\left(\ind G\right)$
for $H$ an induced subgraph of $G$. Thus, for example, finding an
induced subgraph of $G$ whose independence complex is a $d$-dimensional
sphere would show that $\reg\left(R/\edgeideal(G)\right)=\reg_{k}\left(\ind G\right)\geq d+1$. 

Such bounds often do not depend on the choice of field $k$ that we
work over, and in such cases we will suppress $k$ from our notation.\medskip{}

Recall that an \emph{induced matching} in a graph $G$ is a matching
which forms an induced subgraph of $G$, and that $\im G$ denotes
the number of edges in a largest induced matching. Induced matchings
have a considerable literature, see e.g. \cite{Abueida/Busch/Sritharan:2010,Cameron:1989,Cameron:2004,Faudree/Gyarfas/Schelp/Tuza:1989,Golumbic/Lewenstein:2000}. 

The following is essentially due to Katzman; we will give a short
geometric proof.
\begin{lem}
\label{lem:InducedMatchingLB}\emph{(Katzman \cite[Lemma 2.2]{Katzman:2006})}
For any graph $G$, we have $\reg\left(R/\edgeideal(G)\right)\geq\im G$.\end{lem}
\begin{proof}
\noindent Let $m=\im G$, so that $G$ has $mK_{2}$ as an induced
subgraph. Notice that if $H$ is the disjoint union of subgraphs $H_{1}$
and $H_{2}$, then $\ind(H)$ is the simplicial join $\ind(H_{1})*\ind(H_{2})$.
Thus, the independence complex of the disjoint union of $m$ edges
is the $m$-fold join of $0$-spheres, hence an $(m-1)$-sphere. (It
is the boundary complex of an $(m-1$)-dimensional cross-polytope.)
The result follows.
\end{proof}
A more general result follows immediately from the Künneth formula
in algebraic topology \cite[(9.12)]{Bjorner:1995}:
\begin{lem}
\label{lem:RegularityOfJoin}For any field $k$ and simplicial complexes
$\Delta_{1}$ and $\Delta_{2}$, we have 
\[
\reg_{k}\left(\Delta_{1}*\Delta_{2}\right)=\reg_{k}\Delta_{1}+\reg_{k}\Delta_{2}.
\]

\end{lem}
\noindent In the context of edge ideal quotients, if $G_{1}$ and
$G_{2}$ are any two graphs then over any field $k$, then for their
disjoint union $G_{1}\amalg G_{2}$ we have 
\begin{equation}
\reg\left(R/\edgeideal(G_{1}\amalg G_{2})\right)=\reg\left(R/\edgeideal(G_{1})\right)+\reg\left(R/\edgeideal(G_{2})\right).\label{eq:DisjointUnionReg}
\end{equation}
Thus, Lemma \ref{lem:InducedMatchingLB} is the special case where
we take the disjoint union of graphs with a single edge.\medskip{}

\noindent Lemmas \ref{thm:MainLemma} and \ref{lem:InducedMatchingLB}
admit the simple combined statement that for any graph $G$ we have
\begin{equation}
\im G\leq\reg\left(R/\edgeideal\left(G\right)\right)\leq\cochordcov G.\label{eq:bounds}
\end{equation}
Both inequalities can both be strict, as the interested reader can
quickly see by examination of $C_{5}$ and $C_{7}$. Indeed, it follows
easily that regularity can be arbitrarily far from both $\im G$ and
$\cochordcov G$:
\begin{prop}
\label{prop:UB-LB_gap} For any nonnegative integers $r,s$ there
is a graph $G$ such that 
\[
\im G=\reg\left(R/\edgeideal(G)\right)-r\quad\mbox{and}\quad\cochordcov G=\reg\left(R/\edgeideal(G)\right)+s.
\]
\end{prop}
\begin{proof}
Consider $r$ copies of $C_{5}$ disjoint union with $s$ copies of
$C_{7}$.
\end{proof}
\noindent Another relevant construction can be found in Lemma \ref{lem:WhiskeredGraph}
and the discussion following.

\medskip{}

More generally, Kozlov calculated the homotopy type of the independence
complexes of paths and cycles \cite[Propositions 4.6 and 5.2]{Kozlov:1999},
from which the following is immediate:
\begin{prop}
\label{pro:PathsAndCycles}$\reg\left(R/\edgeideal(C_{n})\right)=\reg\left(R/\edgeideal(P_{n})\right)=\left\lfloor \frac{n+1}{3}\right\rfloor $
for $n\geq3$.
\end{prop}
\noindent (Regularity of $R/\edgeideal(P_{n})$ was also calculated
in \cite{Bouchat:2010} using purely algebraic methods.)

It is easy to see that the regularity is equal to the lower bound
of Lemma \ref{lem:InducedMatchingLB} in the $P_{n}$ case, and in
the $C_{n}$ case when $n\not\equiv2$ (mod 3); but that $\reg(\ind(C_{3i+2}))=i+1=\im(C_{3i+2})+1$. 

Since the graph formed by two disjoint edges is not co-chordal, we
see that co-chordal subgraphs of $P_{n}$ and $C_{n}$ (for $n\geq5$)
are paths with at most 3 edges. Thus, regularity is equal to the upper
bound of Theorem \ref{thm:MainLemma} in the $P_{n}$ case, and in
the $C_{n}$ case when $n\not\equiv1$ (mod 3); but for $i>1$ we
have $\reg(\ind\left(C_{3i+1}\right))=i=\cochordcov\left(C_{3i+1}\right)-1$.\medskip{}

By combining Proposition \ref{pro:PathsAndCycles} with Lemma \ref{lem:RegularityOfJoin},
we can somewhat improve the induced matching lower bound of Lemma
\ref{lem:InducedMatchingLB}:
\begin{cor}
\label{cor:IndmatchAndCycleLB}If a graph $G$ has an induced subgraph
$H$ which is the disjoint union of edges and cycles 
\[
H\cong mK_{2}\,\amalg\,\coprod_{j=1}^{n}C_{3i_{j}+2}
\]
then $\reg\left(R/\edgeideal(G)\right)\geq m+n+\sum_{j=1}^{n}i_{j}$.
\end{cor}

\section{\label{sec:Applications}Applications}

We can recover, and in some cases improve, several of the upper bounds
for regularity in the combinatorial commutative algebra literature
by combining Theorem~\ref{thm:MainLemma} with covering results from
the graph theory literature. Theorem \ref{thm:MainLemma} thus seems
to capture an essential connection between Castelnuovo-Mumford regularity
and pure graph-theoretic invariants.

\subsection{Split covers}

Although co-chordal covers per se have not been a topic of frequent
study, there are many results in the graph theory literature concerning
the $\mathcal{F}$-cover number of graphs for various subfamilies
of co-chordal graphs. We will review several of these with an eye
to regularity.\smallskip{}

A \emph{split graph} is a graph $H$ such that $V(H)$ can be partitioned
into a clique and an (induced) independent set. It is easy to see
that such graphs are both chordal and co-chordal; see e.g. \cite[Chapter 5]{Mahadev/Peled:1995}
for additional background. Covering the edges of $G$ with split graphs
allows us to prove Theorem \ref{thm:SplitCovering}.
\begin{proof}[Proof of Theorem \ref{thm:SplitCovering}.]
\emph{ (Essentially e.g. \cite[Lemma 7.5.2]{Mahadev/Peled:1995}).}
Let $H_{i}$ be the subgraph consisting of all edges incident to at
least one vertex in $J_{i}$. Each $H_{i}$ can be partitioned as
the clique on $J_{i}$ together with the independent set $V(G)\setminus V(J_{i})$.
Therefore, each $H_{i}$ is a split graph. Thus $H_{1},\dots,H_{s}$
is a split graph covering, hence a co-chordal covering. The result
follows by Theorem~\ref{thm:MainLemma}.
\end{proof}
To help clarify the meaning of the condition in Theorem \ref{thm:SplitCovering},
we notice that when $J_{0}=\emptyset$, the sets $J_{1},\dots,J_{s}$
are exactly an $s$-coloring of $\overline{G}$. 

However, the bound $\reg\left(R/\edgeideal(G)\right)\leq\chi(\overline{G})$
resulting from the $J_{0}=\emptyset$ case of Theorem \ref{thm:SplitCovering}
is trivial. Indeed, this bound follows from the inequalities $\chi(\overline{G})\geq\alpha(G)$
and $\alpha(G)\geq\reg\left(R/\edgeideal(G)\right)$. (The latter
is immediate by Hochster's formula, as discussed in Section \ref{sec:Lower-bounds},
since $\alpha(G)=\dim\ind\left(G\right)+1$ and $\tilde{H}_{i}(\Delta)$
always vanishes above $\dim\Delta$.)

\medskip{}

The proof of Theorem \ref{thm:SplitRecursiveFormulation} is entirely
similar:
\begin{proof}[Proof of Theorem \ref{thm:SplitRecursiveFormulation}.]
 Let $H$ consist of all edges incident to $J$. Then $H$ is a split
graph, with $E(G)=E(H)\cup E(G\setminus J)$, and the result follows
from (\ref{eq:Kalai-Meshulam_EdgeIdeal}).
\end{proof}
We now recall two results of Hà and Van Tuyl, for which we will give
new proofs via Theorem \ref{thm:SplitCovering}. The \emph{matching
number} of a graph $G$, denoted $\nu(G)$, is the size of a maximum
matching; that is, the maximum number of pairwise disjoint edges.
\begin{thm}
\emph{\label{thm:MaxlMatching} (Hà and Van Tuyl }\cite[Theorem 6.7]{Ha/VanTuyl:2008}\emph{)
}For any graph $G$, we have $\reg\left(R/\edgeideal(G)\right)\leq\nu(G)$.\end{thm}
\begin{proof}
This is the special case of Theorem \ref{thm:SplitCovering} where
$J_{1},\dots,J_{s}$ is a maximum size family of 2-cliques. 
\end{proof}
An easy (stronger) corollary of Theorem \ref{thm:SplitCovering} is
that $\reg(R/\edgeideal(G))$ is at most the size of a minimum maximal
matching. Indeed, we can regard Theorem \ref{thm:SplitCovering} as
it is stated to be a strong generalization of Theorem \ref{thm:MaxlMatching}.

We also give a new proof for:
\begin{thm}
\emph{\label{thm:ChordalGraph} (Hà and Van Tuyl }\cite[Corollary 6.9]{Ha/VanTuyl:2008}\emph{)
}If $G$ is a chordal graph, then $\reg\left(R/\edgeideal(G)\right)=\im G$.\end{thm}
\begin{proof}[Proof (of Theorem \ref{thm:ChordalGraph})]
Cameron \cite{Cameron:1989} observed that a chordal graph $G$ has
split cover number (as in Theorem \ref{thm:SplitCovering}) equal
to $\im G$; the result follows by (\ref{eq:bounds}).
\end{proof}

\subsection{Weakly chordal graphs, and techniques for finding co-chordal covers}

We can considerably extend Theorem \ref{thm:ChordalGraph} by considering
more general covers. A graph $G$ is \emph{weakly chordal} if every
induced cycle in both $G$ and $\overline{G}$ has length at most
4. (It is straightforward to show that a chordal graph is weakly chordal.)
\begin{thm}
\label{thm:WeaklyChordal} If $G$ is a weakly chordal graph, then
$\reg\left(R/\edgeideal(G)\right)=\im G$.\end{thm}
\begin{proof}
Busch, Dragan, and Sritharan \cite[Proposition 3]{Busch/Dragan/Sritharan:2010}
show that $\im G=\cochordcov G$ for any weakly chordal graph $G$.
(Abueida, Busch, and Sritharan \cite[Corollary 1]{Abueida/Busch/Sritharan:2010}
earlier showed the same result under the additional assumption that
$G$ is bipartite.)
\end{proof}
The essential technique introduced in \cite{Cameron:1989} and further
developed in \cite{Abueida/Busch/Sritharan:2010,Busch/Dragan/Sritharan:2010}
is to examine a derived graph $G^{*}$, with vertices corresponding
to the edges of $G$, and two edges adjacent unless they form an induced
matching in $G$. Thus, an independent set of $G^{*}$ corresponds
to an induced matching of $G$. (In graph-theoretic terms, $G^{*}$
is the square of the line graph of $G$.)

In a weakly chordal \cite{Busch/Dragan/Sritharan:2010} (chordal \cite{Cameron:1989},
chordal bipartite \cite{Abueida/Busch/Sritharan:2010}) graph, these
papers show that 
\begin{enumerate}
\item [i) ]$G^{*}$ is perfect, so that there is a partition of the vertices
of $G^{*}$ into $\alpha(G^{*})$ cliques, and 
\item [ii) ]that the subgraph of $G$ corresponding to a maximal clique
of $G^{*}$ is co-chordal. 
\end{enumerate}
The equality of $\im G$ and $\cochordcov G$ follows.

We use a modification of this approach to prove Theorem \ref{thm:WellCoveredBipariteCochord}
below.

\subsection{Biclique and chain graph covers}

Following our terminology from Section \ref{sub:Notation}, the biclique
cover number of a graph $G$ is the minimum number of bicliques (complete
bipartite graphs) required to cover the edges of $G$. As a complete
bipartite graph $K_{m,n}$ is clearly co-chordal, the biclique cover
number is an upper bound for $\cochordcov G$. More generally, it
is straightforward to show that a bipartite graph $G$ is co-chordal
if and only $\im G=1$. Bipartite co-chordal graphs have been called
\emph{chain graphs}.

Recall that a graph is \emph{well-covered} if every maximal independent
set has the same cardinality. Kumini showed:
\begin{thm}
\emph{(Kumini \cite{Kummini:2009})} \label{thm:RegWellcoveredBipartite}
If $G$ is a well-covered bipartite graph, then $\reg\left(R/\edgeideal(G)\right)=\im G$.
\end{thm}
\noindent We recover Theorem \ref{thm:RegWellcoveredBipartite} as
a corollary of the following chain graph covering result:
\begin{thm}
\label{thm:WellCoveredBipariteCochord} If $G$ is a well-covered
bipartite graph, then $\im G=\cochordcov G$.
\end{thm}
In order to prove Theorem \ref{thm:WellCoveredBipariteCochord}, we
will need two lemmas. First, well-covered bipartite graphs have long
been known to admit a simple characterization:
\begin{lem}
\emph{\label{lem:WellCoveredBipChar} (Ravindra \cite{Ravindra:1977},
Favaron \cite{Favaron:1982}; see also Villarreal} \cite{Villarreal:2007}\emph{)}
If $G$ is a well-covered bipartite graph with no isolated vertices,
then $G$ has a perfect matching. Moreover, in every perfect matching
$M$ of $G$ the neighborhood of any edge in $M$ is complete bipartite.
\end{lem}
We will also need the following technical lemma. Two edges are \emph{incident}
if they share a vertex; in particular, we consider an edge to be incident
to itself.
\begin{lem}
\label{lem:CochordalCovWCBip} Let $G$ be a well-covered bipartite
graph, and $M$ a perfect matching in $G$. Let $M_{0}$ be a subset
of $M$ so that no pair of edges in $M_{0}$ form an induced matching
in $G$. Then the subgraph $H$ of $G$ consisting of all edges incident
to $M_{0}$ has $\im H=1$, and is in particular co-chordal.\end{lem}
\begin{proof}
Since the neighborhood of any edge in $M$ is complete bipartite,
it suffices to show that if $e$ is an edge of $H$ and $c_{0}$ an
edge of $M_{0}$, then $e$ and $c_{0}$ do not form a $2K_{2}$;
that is, that there is some edge of $G$ incident to both $e$ and
$c_{0}$.

If $e\in M_{0}$ then this is immediate by the hypothesis. Otherwise,
$e=\{x,y\}$ where $y$ is in some edge $c_{1}=\{y,z\}$ of $M_{0}$.
By the hypothesis on $M$, either $y$ or $z$ is in some edge $b$
incident to $c_{0}$. If $y\in b$ then we are done. Otherwise, $b=\{z,w\}$
with $w\in c_{0}$. But then $w$ and $x$ are both neighbors of $c_{1}$,
hence adjacent by Lemma \ref{lem:WellCoveredBipChar}.
\end{proof}

\begin{proof}[Proof of Theorem \ref{thm:WellCoveredBipariteCochord}]

Assume without loss of generality that $G$ has no isolated vertices,
and let $M$ be a perfect matching, as guaranteed to exist by Lemma
\ref{lem:WellCoveredBipChar}. We construct a new graph $M^{*}$ with
vertices consisting of the edges of $M$, and with two vertices adjacent
unless they form an induced matching in $G$. Thus, $M^{*}$ is an
induced subgraph of the graph $G^{*}$ from the discussion following
Theorem \ref{thm:WeaklyChordal}.

Any independent set in $M^{*}$ still corresponds to an induced matching
of $G$, so that $\alpha(M^{*})\leq\im G$. On the other hand, if
$K^{*}$ is a clique in $M^{*}$, then Lemma \ref{lem:CochordalCovWCBip}
gives the subgraph of all incident edges to be co-chordal. Since every
edge in $G$ is incident to at least one edge of $M$, we get that
$\cochordcov G\leq\chi(\overline{M^{*}})$.

But Kumini shows \cite[Discussion 2.8]{Kummini:2009} that the graph
obtained from $M^{*}$ by identifying pairs of vertices $v$ and $w$
with $N[v]=N[w]$ is a comparability graph, hence perfect; so $M^{*}$
is perfect by e.g. Diestel \cite[Lemma 5.5.5]{Diestel:2005}. Hence,
we have that $\alpha(M^{*})=\chi(\overline{M^{*}})$, and the result
follows.
\end{proof}
We remark that in Theorem \ref{thm:WeaklyChordal}, we apply a result
from the graph theory literature to prove a new result on regularity;
while in Theorem \ref{thm:WellCoveredBipariteCochord}, a result from
combinatorial commutative algebra guides us to a new min-max result
on well-covered bipartite graphs.

\subsection{Co-interval covers and boxicity}

An \emph{interval graph} is a graph with vertices corresponding to
some set of intervals in $\mathbb{R}$, and edges between pairs of
intervals that have non-empty intersection. A \emph{co-interval graph}
is the complement of an interval graph. Interval graphs are exactly
the chordal graphs which can be represented as the incomparability
graph of a poset. See \cite{Mahadev/Peled:1995} for general background
on such graphs.

The \emph{boxicity} of $G$, denoted $\boxicity G$, is the co-interval
cover number of $\overline{G}$. (The original formulation of boxicity
was somewhat different, and the connection with covering is made in
\cite{Cozzens/Roberts:1983}.) Thus by Theorem \ref{thm:MainLemma}
we have that $\reg\left(R/\edgeideal(G)\right)\leq\boxicity\overline{G}$.\medskip{}

Since a planar graph $G$ contains no $K_{5}$ subgraph, we have that
$\reg\left(R/\edgeideal(\overline{G})\right)\leq\dim\ind\left(\overline{G}\right)+1=\alpha\left(\overline{G}\right)\leq4$.
The literature on boxicity yields a stronger result:
\begin{prop}
\label{prop:CoplanarBound} If $G$ is a planar graph, then $\reg\left(R/\edgeideal(\overline{G})\right)\leq3$.
This upper bound is the best possible.\end{prop}
\begin{proof}
Thomassen \cite{Thomassen:1986} proves that $\boxicity G\leq3$.
To see the bound is best possible, notice that the complement of $3K_{2}$
(that is, the graph consisting of 3 disjoint edges) is the 1-skeleton
of the octahedron, which is well-known to be planar.
\end{proof}
By way of contrast, we remark that the proof of Proposition \ref{prop:UB-LB_gap}
shows that if $G$ is a planar graph, then $\reg\left(R/\edgeideal(G)\right)$
may be arbitrarily large.

\subsection{Very well-covered graphs}

In this subsection we present a negative result. A graph is \emph{very
well-covered} if it is well-covered and $\alpha(G)=\left|V\right|/2$.
It is obvious that every well-covered bipartite graph is very well-covered.
Mahmoudi et al.~\cite{Mahmoudi/Mousivand/Crupi/Rinaldo/Terai/Yassemi:2011}
generalized Theorem \ref{thm:RegWellcoveredBipartite} to show:
\begin{thm}
\emph{\label{thm:VerywellcoveredReg} (Mahmoudi, Mousivand, Crupi,
Rinaldo, Terai, and Yassemi }\cite{Mahmoudi/Mousivand/Crupi/Rinaldo/Terai/Yassemi:2011}\emph{)}\\
\emph{ }If $G$ is a very well-covered graph, then $\reg\left(R/\edgeideal(G)\right)=\im G$.
\end{thm}
\noindent We will demonstrate, however, that the gap between $\im G$
and $\cochordcov G$ can be arbitrarily large for very well-covered
graphs. In particular, the proof via (\ref{eq:bounds}) of Theorem
\ref{thm:RegWellcoveredBipartite} cannot be extended to prove Theorem
\ref{thm:VerywellcoveredReg}.\medskip{}

If $G$ is a graph on $n$ vertices, then let $W(G)$ be the graph
on $2n$ vertices obtained by adding a \emph{pendant} (an edge to
a new vertex of degree 1) at every vertex of $G$. This construction
has been previously studied in the context of graphs with Cohen-Macaulay
edge ideals \cite{Villarreal:1990}, where it has been referred to
as \emph{whiskering}; and has been studied in the graph theory literature
as a \emph{corona} \cite{Frucht/Harary:1970}. Because the pendant
vertices form a maximal independent set, it is immediate that $W(G)$
is very well-covered. 
\begin{lem}
\label{lem:WhiskeredGraph} For any graph $G$, we have $\im W(G)=\alpha(G)$
and $\cochordcov W(G)=\chi(\overline{G})$.\end{lem}
\begin{proof}
For the first equality, we notice that if an induced matching of $W(G)$
contains an edge $\{v,w\}$ of $G$, then we can get a new induced
matching by replacing $\{v,w\}$ with the pendant edge at $v$. Since
a collection of pendant edges forms an induced matching if and only
if the corresponding collection of vertices of $G$ is independent,
the statement follows.

For the second equality, we first notice that a coloring of $\overline{G}$
partitions the vertices of $G$ into cliques, inducing a covering
of $W(G)$ by split graphs (as in Theorem \ref{thm:SplitCovering}).
Hence $\cochordcov W(G)\leq\chi(\overline{G})$. On the other hand,
any co-chordal cover $\left\{ H_{i}\right\} $ of $W(G)$ in particular
covers the pendant edges, and two pendant edges form an induced matching
if the corresponding vertices of $G$ are not connected. Hence a co-chordal
cover induces a covering of the vertices of $G$ by cliques, and thus
$\cochordcov W(G)\geq\chi(\overline{G})$, as desired.
\end{proof}
But then, for example, we have $\im W(C_{5})=2$ and $\cochordcov W(C_{5})=3$.
Moreover, it is well-known that the gap between the clique number
and chromatic number of $\overline{G}$ can be arbitrarily large,
even if $\omega(\overline{G})=\alpha(G)=2$. (See e.g. \cite[Theorem 5.2.5]{Diestel:2005}.)
Hence, the gap between $\im W(G)$ and $\cochordcov W(G)$ can also
be arbitrarily large. \medskip{}

Van Tuyl \cite{VanTuyl:2009} has shown an analogue to Theorem \ref{thm:WellCoveredBipariteCochord}:
that if $G$ is a bipartite graph such that $R/\edgeideal(G)$ is
sequentially Cohen-Macaulay, then $\reg\left(R/\edgeideal(G)\right)=\im G$.
(See his paper \cite{VanTuyl:2009} for definitions and background.)
The following example, however, shows that $\im G$ and $\reg\left(R/\edgeideal(G)\right)$
may also be strictly less than $\cochordcov G$ in this situation. 
\begin{example}
Let $G$ be obtained from $C_{6}$ by attaching a pendant to vertices
$x_{1},x_{2},x_{3},$ and $x_{4}$. It is easy to see from the conditions
given in \cite{VanTuyl:2009} that $\ind G$ is sequentially Cohen-Macaulay.
But an approach similar to that in Lemma \ref{lem:WhiskeredGraph}
will verify that $\im G=2$, while $\cochordcov G=3$.
\end{example}

\subsection{Computational complexity}

An immediate consequence of Lemma \ref{lem:WhiskeredGraph} is that
calculating $\reg\left(R/\edgeideal(G)\right)$ from the graph $G$
is computationally hard:
\begin{cor}
\label{cor:NPcomplete} Given $G$, calculating $\reg\left(R/\edgeideal(G)\right)$
is NP-hard, even if $G$ is very well-covered.\end{cor}
\begin{proof}
One can construct $W(G)$ from $G$ in polynomial time, and $\reg\left(R/\edgeideal(W(G))\right)=\im W(G)=\alpha(G)$.
But checking whether $\alpha(G)\geq C$ is well-known to be NP-complete! 
\end{proof}
Since computing the independence complex of $G$ is already NP-hard,
and as it is hard to imagine finding regularity without computing
the independence complex, Corollary \ref{cor:NPcomplete} is perhaps
not too surprising. It might be of more interest to find the computational
complexity of computing $\reg\left(R/\edgeideal(G)\right)$ from $\ind G$.\smallskip{}

We remark that many of the results we have referenced are from the
computer science literature, and efficient algorithms for finding
$\im G$ and $\cochordcov G$ in special classes of graphs are a main
interest of \cite{Abueida/Busch/Sritharan:2010,Busch/Dragan/Sritharan:2010,Cameron:1989}
and other papers. In particular, given a weakly chordal graph $G$,
we can calculate $\reg(R/\edgeideal(G))=\im G$ in polynomial time
\cite[Corollary 8]{Busch/Dragan/Sritharan:2010}.

In general graphs, however, computing $\im G$ or $\cochordcov G$
is NP-hard: It follows from e.g. the proof of Corollary \ref{cor:NPcomplete}
that determining whether $\im G\geq C$ is NP-complete; while Yannakakis
showed \cite{Yannakakis:1982} that determining whether $\cochordcov(G)\leq C$
is NP-complete. The corresponding problem for split graph covers (as
in Theorem~\ref{thm:SplitCovering}) is also NP-complete \cite{Chernyak/Chernyak:1991}.
An overview of these and similar hardness results can be found in
\cite[Chapter 7]{Mahadev/Peled:1995}.

\subsection{Questions on claw-free graphs}

Nevo \cite{Nevo:2011} showed that if $G$ is a ($2K_{2}$, claw)-free
graph, then $\reg\left(R/\edgeideal(G)\right)\leq2$. Dao, Huneke,
and Schweig \cite{Dao/Huneke/Schweig:2013} have recently given an
alternate proof.  Can the same be shown using Theorem \ref{thm:MainLemma}?
\begin{question}
\label{thm:2K2+ClawFree}If $G$ is ($2K_{2}$, claw)-free, then is
$\cochordcov G\leq2$?
\end{question}
We notice that a cover by split graphs will not suffice: for example,
the Petersen graph $P$ has girth 5, hence $\overline{P}$ is ($2K_{2}$,
claw)-free. But it is easy to verify that no 2 cliques in $\overline{P}$
satisfy the condition of Theorem \ref{thm:SplitCovering}. 

András Gyárfás points out {[}personal communication{]} that in \cite[Problem 5.7]{Gyarfas:1987}
he has asked whether every graph $G$ with $\cochordcov G=2$ has
$\chi(\overline{G})$ bounded by some function of $\alpha(G)$. We
observe that the complement of a graph with girth $\geq5$ is ($2K_{2}$,
claw)-free, with $\alpha(G)=2$. Since a graph with girth $\geq5$
can have arbitrarily large chromatic number \cite[Theorem 5.2.5]{Diestel:2005},
a positive answer to Question \ref{thm:2K2+ClawFree} would imply
a negative answer to Gyárfás' question. 

If the answer to Question \ref{thm:2K2+ClawFree} is negative, then
the following might still be of interest:
\begin{question}
\label{thm:2K_2-freeCoverClaw-free} If $G$ is claw-free, then does
$G$ have a ($2K_{2}$,claw)-free cover by at most $\im G$ subgraphs?
\end{question}
\noindent If Question \ref{thm:2K_2-freeCoverClaw-free} has a positive
answer, then a direct application of (\ref{eq:Kalai-Meshulam_EdgeIdeal})
would then imply that for a claw-free graph $G$ we have $\reg\left(R/\edgeideal(G)\right)\leq2\cdot\im G$.

\emph{After acceptance of the paper, Shahab Haghi and Siamak Yassemi
pointed out to me by email {[}private communication{]} that Question
\ref{thm:2K_2-freeCoverClaw-free} has a negative answer for the cyclic
graph $C_{8}$. So far as I am aware, the question remains open as
to whether $\reg(R/I(G))\leq2\cdot\im G$ for any claw-free graph
$G$.}

\section*{Acknowledgements}

I am grateful to R.~Sritharan for making me aware of the relevance
of \cite{Abueida/Busch/Sritharan:2010} and \cite{Busch/Dragan/Sritharan:2010},
as well as for several stimulating conversations. Chris Francisco,
András Gyárfás, Huy Tài Hà, Craig Huneke, and Adam Van Tuyl have made
helpful comments and suggestions. I have benefited greatly from the
advice and encouragement of John Shareshian. 

\bibliographystyle{hamsplain}
\bibliography{4_Users_russw_Documents_Research_Master}

\end{document}